\documentclass[10pt,a4paper,reqno]{amsart}
\usepackage{amssymb,amsmath,amsthm}
\usepackage[T1]{fontenc}
\usepackage[dvips]{graphicx}

\theoremstyle{plain}

\newtheorem{corollary}{Corollary}
\newtheorem{lemma}{Lemma}
\newtheorem{proposition}{Proposition}

\newtheorem*{theorem A}{Theorem A}
\newtheorem*{theorem B}{Theorem B}
\newtheorem*{corollary A}{Corollary A}
\newtheorem*{corollary B}{Corollary B}
\theoremstyle{definition}

\newtheorem{rem}{Remark}
\newtheorem{example}{Example}
\newtheorem*{proof A}{Proof of Theorem A}
\newtheorem*{proof B}{Proof of Theorem B}
\newtheorem*{proof AA}{Proof of Corollary A}

\DeclareMathOperator*{\hd}{\mathrm{dim}_\mathrm{H}}

\begin{document}

\title{Hausdorff dimension of certain random self--affine fractals}
\maketitle \centerline{\small{NUNO LUZIA}}
\centerline{\emph{UFRJ, Rio de Janeiro, Brasil}}
\centerline{\emph{e-mail address}: luzia@impa.br} \vspace{24pt}
\centerline{\scriptsize ABSTRACT}
\text{}\\
\begin{footnotesize}
In this work we are interested in the self--affine fractals studied by Gatzouras and Lalley \cite{5} and by the author
\cite{8} which generalize the famous \emph{general Sierpi\'nski carpets} studied by Bedford \cite{2} and McMullen
\cite{10}. We give a formulae for the Hausdorff dimension of sets which are \emph{randomly} generated using a finite
number of self--affine transformations each one generating a fractal set as mentioned before, with some technical
hypotheses. The choice of the transformation is random according to a Bernoulli measure. The formulae is given in terms
of the variational principle for the dimension.
\end{footnotesize}
\text{}
\text{}\\\\

\section{Introduction}
A main difficulty in calculating Hausdorff dimension is the
phenomenon of \emph{non-conformality} which arises when we have
several rates of expansion. In the $1$-dimensional (conformal)
setting the computation of Hausdorff dimension is possible, at
least in the \emph{uniformly expanding} context, due to the
thermodynamic formalism introduced by Sinai-Ruelle-Bowen (see \cite{3}
and \cite{12}). The problem of calculating Hausdorff dimension in the
non-conformal setting was first considered by Bedford \cite{2} and
McMullen \cite{10}. They showed independently that for the class of
transformations called \emph{general Sierpi\'nski carpets}, there
exists an ergodic measure of full Hausdorff dimension. Following
these works several extensions have been made, e.g. in \cite{1}, \cite{5}, \cite{6},
\cite{7}, \cite{8} and \cite{9}. In this work we are interested in the self--affine fractals studied in \cite{5} and
\cite{8}, which we call \emph{self--affine Sierpi\'nski carpets}, namely we compute the Hausdorff dimension of a random
version of these sets, with some technical hypotheses.

Let $m\in\mathbb{N}$ and consider the symbolic space $\mathcal{I}=\{1,...,m\}^\mathbb{N}$ equipped with a Bernoulli
measure given by the probability vector ${\bf p}=(p_1,...,p_m)$ (we assume $p_i>0$ for all $i$).
This will be our space to which \emph{random} refers to.

For the moment let us fix $i\in\{1,...,m\}$ and remember the construction of \emph{self--affine Sierpi\'nski carpets}.
Let $S_1, S_2, ..., S_r$ be contractions of $\mathbb{R}^2$. Then
there is a unique nonempty compact set $\Lambda$ of $\mathbb{R}^2$
such that
\[
   \Lambda=\bigcup_{l=1}^{r} S_l(\Lambda) .
\]
We will refer to $\Lambda$ as the limit set of the semigroup
generated by $S_1, S_2, ..., S_r$. Now consider the sets
$\Lambda$ which are limit sets of the semigroup generated by the
$2$-dimensional mappings $A_{ijk}$ given by
\[
 A_{ijk}=\begin{pmatrix} a_{ijk} & 0 \\
  0 & b_{ij}  \end{pmatrix}x+
    \begin{pmatrix} c_{ijk} \\ d_{ij}
  \end{pmatrix}
\]
for $(j,k)\in \mathcal{J}_i$, where
\begin{align*}
 \mathcal{J}_i=\{(j,k):\; 1\le j\le m_{i},\; 1\le k\le m_{ij} \}
\end{align*}
is a finite index set. We assume $0<a_{ijk}\le b_{ij}<1$,
$\sum_{j=1}^{m_i} b_{ij} \le 1$ and $\sum_{k=1}^{m_{ij}} a_{ijk} \le 1$.
Also, $0\le d_{i1}<d_{i2}<...<d_{im_i}<1$ with $d_{ij+1}-d_{ij}\ge b_{ij}$ and $1-d_{im_i}\ge b_{im_i}$, and
$c_{ij1}<c_{ij2}<...<c_{ijm_{ij}}<1$ with $c_{ij(k+1)}-c_{ijk}\ge a_{ijk}$
and $1-c_{ijm_{ij}}\ge a_{ijm_{ij}}$. These hypotheses guarantee that the
rectangles
\[
     R_{ijk}=A_{ijk}([0,1]^2)
\]
have interiors that are pairwise disjoint, with edges parallel to
the $x$- and $y$-axes, are arranged in ``rows''  of height $b_{ij}$ and have width $a_{ijk}$.
Geometrically, $\Lambda$ is the limit (in the Hausdorff metric), or the intersection, of
\emph{$n$-approximations}: the 1-approximation consisting of the rectangles $R_{ijk}$, the
2-approximation consisting in replacing each rectangle of the 1-approximation by an affine copy of the
2-approximation, and so on. Formally,
\[
   \Lambda= \bigcap_{n=1}^\infty \bigcup_{(j_1,k_1),...,(j_n,k_n)\in\mathcal{J}_{i}}
   A_{ij_1k_1}\circ\cdots\circ A_{ij_nk_n}([0,1]^2).
\]

Now we want to give a random version of this construction in such a way that at each step of the approximation we are
allowed to change the number $i\in\{1,...,m\}$. More precisely, given ${\bf i}=(i_1, i_2,...)\in\mathcal{I}$, we consider
the \emph{random set} given by
\[
   \Lambda_{\bf i}= \bigcap_{n=1}^\infty \bigcup_{(j_1,k_1)\in\mathcal{J}_{i_1},...,(j_n,k_n)\in\mathcal{J}_{i_n}}
   A_{i_1j_1k_1}\circ\cdots\circ A_{i_nj_nk_n}([0,1]^2).
\]
Of course this construction generalizes the previous one by putting ${\bf{i}}=(i,i,...)$.

We will need the following \emph{generic} hypothesis on the numbers $a_{ijk}$. For each $t\in [0,1]$,
there exist $1\le i\le m$ and $1\le j<j'\le m_i$ such that
\begin{equation}\label{hip}
  \sum_{k=1}^{m_{ij}} a_{ijk}^t \ne \sum_{k=1}^{m_{ij'}} a_{ij'k}^t.
\end{equation}
We will also need one of the following \emph{robust} hypotheses. Let $\varepsilon>0$. For each $i\in\{1,...,m\}$,
there exist numbers $0<a_i\le b_i<1$ such that
\begin{equation}\label{hiptec1}
  (1+\varepsilon)^{-1}<\frac{b_{ij}}{b_{i}}< 1+\varepsilon,\quad (1+\varepsilon)^{-1}<\frac{a_{ijk}}{a_{i}}< 1+\varepsilon,
\end{equation}
or
\begin{equation}\label{hiptec2}
  \frac{b_{ij}}{a_{ijk}}< 1+\varepsilon,
\end{equation}
for all $(j,k)\in\mathcal{J}_i$.

\emph{Notation}: $\hd\Lambda$ stands for the Hausdorff dimension of a set $\Lambda$.

\begin{theorem A}
There exists $\varepsilon>0$ such that if (\ref{hip}) and (\ref{hiptec1}) or (\ref{hip}) and (\ref{hiptec2})
are satisfied then
\begin{equation} \label{form}
 \hd \Lambda_{\bf i}=\sup_{\bf{P}} \left\{ \lambda({\bf{P}})+ t({\bf{P}}) \right\} \quad \text{for ${\bf p}$-a.e. } \bf{i}.
\end{equation}
where ${\bf P}=(p_{ij})$ is a collection of non-negative numbers satisfying
\begin{align*}
  \sum_{j=1}^{m_i} p_{ij}=p_i, \quad i=1,...,m,
\end{align*}
the number $\lambda({\bf{P}})$ is given by
\[
  \lambda({\bf{P}})=\frac{\sum_{i,j} p_{ij} \log p_{ij}-\sum_i p_{i} \log p_{i}}
  {\sum_{i,j} p_{ij}\log b_{ij}},
\]
(by convention: $0\log0=0$)\\
and $t(\bf{P})$ is the unique real in $[0,1]$ satisfying
\[
  \sum_{i,j} p_{ij} \log \left(\sum_{k} a_{ijk}^{t(\bf{P})} \right)=0 .
\]
\end{theorem A}

\begin{rem}
 The number $\lambda({\bf{P}})$ is the Hausdorff dimension in $y$-axis of the
 set of generic points for the distribution ${\bf{P}}|{\bf{p}}$; the number $t(\bf{P})$ is the Hausdorff
 dimension of a typical $1$-dimensional fibre in the $x$-direction relative to the distribution
 ${\bf{P}}$, and is given by a random Moran formula. Also note that by putting $p_i=1$ for some $i$ we get the
 deterministic case, i.e. the Hausdorff dimension of self-affine Sierpi\'nski carpets (satisfying the technical hypotheses).
\end{rem}

It follows from the proof of Theorem A (see Lemma \ref{dimmedida})
that the expression between brackets in (\ref{form}) is the
Hausdorff dimension of a Bernoulli measure $\mu_{\bf{P}}$. Since
the functions ${\bf{P}}\mapsto \lambda({\bf{P}})$ and
${\bf{P}}\mapsto t({\bf{P}})$ are continuous, we obtain the
following.

\begin{corollary A} With the same hypotheses of Theorem A, there exists ${\bf{P}}^*$ such that
\[
    \hd\Lambda_{\bf i}=\hd\mu_{{\bf{P}}^*} \quad \text{for ${\bf p}$-a.e. } \bf{i} .
\]
\end{corollary A}

\begin{example}
Let $k$ be a positive integer and $0<q<1$. Divide the unit square into a grid of $k\times k$ squares with side length $k$.
Consider the random set $\Lambda_q$ constructed as before as the limit of $n$-approximations, such that at each step of
a further approximation we use a transformation that corresponds to a grid where each square has probability $q$ to belong
to this grid. This construction differs from the usual \emph{fractal percolation} in that the randomness is with respect
to the grid (the transformation) and not on each square. This construction is a particular case of ours if one
considers all the possible patterns for the selected squares of the grids. Namely, the grids consisting of $1\le l\le k^2$
squares are in the number of $\binom{k^2}{l}$ and we assign to each of these grids the probability
$a q^l (1-q)^{k^2-l}$, where $a=(1-(1-q)^{k^2})^{-1}$ appears because we are excluding the case which no square
is selected. Then,
\[
     a \sum_{l=1}^{k^2} \binom{k^2}{l} q^l(1-q)^{k^2-l}=1.
\]
Theorem A says that
\[
   \hd\Lambda_q=\frac{1}{\log k} \sup_{\bf P} \Bigl\{-\sum_{i,j} p_{ij}\log p_{ij}+
   \sum_i p_i\log p_i +\sum_{i,j}p_{ij} \log m_{ij}\Bigr\}\quad \text{a.s.}
\]
Using Lagrange multipliers with the restrictions $\sum_j p_{ij}=p_i$ we get that the supremum above is attained at the
probability vector
\[
    p_{ij}=p_i \frac{m_{ij}}{\sum_j m_{ij}}
\]
and
\[
 \hd\Lambda_q= \frac{\sum_i p_i \log \Bigl(\sum_j m_{ij} \Bigl)}{\log k}=
 \frac{a \sum_{l=1}^{k^2} \binom{k^2}{l} q^l(1-q)^{k^2-l} \log l}{\log k} \quad\text{a.s.}
\]
Note that in the above formula $\log k$ corresponds to the Lyapunov exponent and $\log l$ corresponds to the topological
entropy relative to a grid with $l$ squares, so the numerator is the average of the entropies of the transformations.
In this work we are interested in the Hausdorff dimension of random fractals in a non-conformal setting and parameterized
by real numbers.
\end{example}

\section{Basic results}

Here we mention some basic results about fractal geometry and
pointwise dimension. For proofs we refer the reader to the books
\cite{4} and \cite{11}.

We are going to define the Hausdorff dimension of a set
$F\subset\mathbb{R}^n$. The diameter of a set
$U\subset\mathbb{R}^n$ is denoted by $|U|$. If $\{U_i\}$ is a
countable collection of sets of diameter at most $\delta$ that
cover $F$, i.e. $F\subset\bigcup_{i=1}^{\infty} U_i$ with
$|U_i|\le\delta$ for each $i$, we say that $\{U_i\}$ is a
$\delta$-cover of $F$. Given $t\ge 0$, we define the
\emph{$t$-dimensional Hausdorff measure of} $F$ as
\[
  \mathcal{H}^t(F)=\lim_{\delta\to 0} \inf \left\{ \sum_{i=1}^{\infty} |U_i|^t: \{U_i\}
  \text{ is a $\delta$-cover of $F$} \right\}.
\]
It is not difficult to see that there is a critical value $t_0$
such that
\[
  \mathcal{H}^t(F)=
  \begin{cases}
    \infty &\text{ if } t<t_0 \\
    0      &\text{ if } t>t_0.
   \end{cases}
\]
We define the \emph{Hausdorff dimension} of $F$, written $\hd F$,
as being this critical value $t_0$.

Let $\mu$ be a Borel probability measure on $\mathbb{R}^n$. The
Hausdorff dimension of the measure $\mu$ was defined by L.-S.
Young as
\[
    \hd\mu=\inf\{\hd F : \mu(F)=1\}.
\]
So, by definition, one has
\[
    \hd F\ge\sup\{\hd\mu : \mu(F)=1\}.
\]
In this paper we are interested in the validity of the opposite
inequality in a dynamical context. In practice, to calculate the
Hausdorff dimension of a measure, it is useful to compute its
\emph{lower pointwise dimension}:
\[
  \underline{d}_\mu(x)=\liminf_{r\to 0} \,\frac{\log \mu(B(x,r))}{\log r},
\]
where $B(x,r)$ stands for the open ball of radius $r$ centered at
the point $x$. The relations between these dimensions are given by
the following propositions.

\begin{proposition}
\text{}
 \begin{enumerate}
  \item
   If \,$\underline{d}_\mu(x)\ge d$ for $\mu$-a.e. $x$ then $\hd \mu \ge d$.
  \item
   If \,$\underline{d}_\mu(x)\le d$ for $\mu$-a.e. $x$ then $\hd \mu \le d$.
  \item
   If \,$\underline{d}_\mu(x)=d$ for $\mu$-a.e. $x$ then $\hd \mu=d$.
 \end{enumerate}
\end{proposition}

\begin{proposition}
 If $\underline{d}_\mu(x)\le d$ for every $x\in F$ then $\hd F\le d$.
\end{proposition}

\section{Proof of Theorem A}
\text{}

\emph{Part 1:} $\hd\Lambda_{\bf i}\ge\sup_{\bf{P}}\{\lambda({\bf{P}})+ t(\bf{P})\}$\\

There is a natural symbolic representation associated with our
system that we shall describe now. Given ${\bf i}=(i_1, i_2,...)$, consider the sequence space
$\Omega_{\bf i}=\prod_{n=1}^{\infty} \mathcal{J}_{i_n}$. Elements of $\Omega_{\bf i}$ will be
represented by $\omega=(\omega_1, \omega_2, ...)$ where
$\omega_n=(j_n, k_n)\in\mathcal{J}_{i_n}$. Given
$\omega\in\Omega$ and $n\in\mathbb{N}$, let $\omega(n)=(\omega_1,
\omega_2, ..., \omega_n)$ and define the \emph{cylinder of order
$n$},
\[
    C_{\omega(n)}^{\bf i}=\{\omega'\in\Omega_{\bf i}: \omega'_l=\omega_l,\,l=1,...,n\},
\]
and the \emph{basic rectangle of order $n$},
\[
 R_{\omega(n)}^{\bf i}=A_{i_1\omega_1} \circ A_{i_2\omega_2} \circ \cdots \circ A_{i_n\omega_n}([0,1]^2).
\]
We have that $(R_{\omega(n)}^{\bf i})_n$ is a decreasing sequence of
closed rectangles having edges with length $\prod_{l=1}^n b_{i_lj_l}$ and $\prod_{l=1}^n a_{i_lj_lk_l}$.
Thus $\bigcap_{n=1}^\infty R_{\omega(n)}^{\bf i}$ consists of a single point which belongs to $\Lambda_{\bf i}$ that we
denote by $\chi_{\bf i}(\omega)$. This defines a continuous and surjective
map $\chi_{\bf i}\colon \Omega_{\bf i}\to\Lambda_{\bf i}$ which is at most 4 to 1, and
only fails to be a homeomorphism when some of the rectangles $R_{ijk}$ have nonempty intersection.

We shall construct probability measures $\mu_{\bf{P}, \bf{i}}$ supported
on $\Lambda_{\bf i}$ with
\[
\hd\mu_{\bf{P}, \bf{i}}=\lambda({\bf{P}})+ t(\bf{P}) \quad \text{for ${\bf p}$-a.e. } \bf{i} .
\]
This gives what we want because $\hd\Lambda_{\bf i}\ge\hd\mu_{\bf{P}, \bf{i}}$.

Let $\tilde{\mu}_{\bf{P}, \bf{i}}$ be the Bernoulli measure on $\Omega_{\bf i}$ such that
\[
    \tilde{\mu}_{\bf{P}, \bf{i}}(C_{\omega(n)}^{\bf i})=\prod_{l=1}^n \frac{p_{i_lj_l}}{p_{i_l}}
    \,\frac{a_{i_lj_lk_l}^{t(\bf{P})}}{\sum_{k} a_{i_lj_lk}^{t(\bf{P})}}.
\]
Let $\mu_{\bf{P}, \bf{i}}$ be the probability measure on $\Lambda_{\bf i}$ which
is the pushforward of $\tilde{\mu}_{\bf{P}, \bf{i}}$ by $\chi_{\bf i}$, i.e.
$\mu_{\bf{P}, \bf{i}}=\tilde{\mu}_{\bf{P}, \bf{i}}\circ\chi_{\bf i}^{-1}$.

For calculating the Hausdorff dimension of $\mu_{\bf{P}, \bf{i}}$ we
shall consider some special sets called \emph{approximate squares}.
Given $\omega\in\Omega_{\bf i}$ and $n\in\mathbb{N}$ such that $n\ge(\log
\min a_{ijk})/(\log \max b_{ij})$, define
\begin{equation}\label{ln}
 L_n(\omega)=\max\left\{ k\ge1: \prod_{l=1}^n b_{i_lj_l}\le
 \prod_{l=1}^k a_{i_lj_lk_l}\right\}
\end{equation}
and the \emph{approximate square}
\[
B_n(\omega)=\left\{
    \omega'\in\Omega_{\bf i}:\:\, j_{l}'=j_l,\: l=1,...,n \text{ and }
     k_{l}'=k_l,\: l=1,...,L_n(\omega)  \right\}.
\]
We have that each approximate square $B_n(\omega)$ is a finite union
of cylinder sets, and that approximate squares are \emph{nested},
i.e., given two, say $B_n(\omega)$ and $B_{n'}(\omega')$, either
$B_n(\omega)\cap B_{n'}(\omega')=\emptyset$ or $B_n(\omega)\subset
B_{n'}(\omega')$ or $B_{n'}(\omega')\subset B_n(\omega)$.
Moreover, $\chi_{\bf i}(B_n(\omega))=\tilde{B}_n(\omega)\cap \Lambda_{\bf i}$
where $\tilde{B}_n(\omega)$ is a closed rectangle in $\mathbb{R}^2$ with
edges parallel to the coordinate axes, with vertical length $\prod_{l=1}^{n} b_{i_lj_l}$ and horizontal length
$\prod_{l=1}^{L_n(\omega)} a_{i_lj_lk_l}$. By
(\ref{ln}),
\begin{equation}\label{Q}
  1\le \frac{\underset{l=1}{\overset{L_n(\omega)}{\prod}} a_{i_lj_lk_l}}
  {\underset{l=1}{\overset{n}{\prod}} b_{i_lj_l}}\le \max\,a_{ijk}^{-1},
\end{equation}
hence the term ``approximate square''. It follows
from (\ref{Q}) that
\begin{equation}\label{factor0}
  \frac{\sum_{l=1}^{L_n(\omega)} \log a_{i_lj_lk_l}}{\sum_{l=1}^n \log b_{i_lj_l}}
  =1+\frac{1}{n} \frac{\sum_{l=1}^{L_n(\omega)} \log a_{i_lj_lk_l}-\sum_{l=1}^{n}
  \log b_{i_lj_l}} {\frac{1}{n} \sum_{l=1}^{n}\log b_{i_lj_l}}\to 1 .
\end{equation}
Also observe that $L_n(\omega)\le n$ and $L_n(\omega)\to\infty$ as $n\to\infty$.

\begin{lemma}\label{dimmedida}
 $\hd \mu_{\bf{P}, \bf{i}}=\lambda({\bf{P}})+ t(\bf{P}) \quad \text{for ${\bf p}$-a.e. } \bf{i}$.
\end{lemma}
\begin{proof}
To calculate the Hausdorff dimension of $\mu_{\bf{P}, \bf{i}}$ we are
going calculate its pointwise dimension and use Proposition 1.
Remember that $\chi_{\bf i}(B_n(\omega))=\tilde{B}_n(\omega)\cap \Lambda_{\bf i}$
where, by (\ref{Q}), $\tilde{B}_n(\omega)$ is ``approximately'' a
ball in $\mathbb{R}^2$ with radius $\prod_{l=1}^n b_{i_lj_l}$, and
that
\[
   \mu_{\bf{P}, \bf{i}}(\tilde{B}_n(\omega))=\tilde{\mu}_{\bf{P}, \bf{i}}(B_n(\omega)).
\]
Also, $\chi_{\bf i}$ is at most $4$ to 1. Taking this into account, by
Proposition 1 together with [11, Theorem 15.3], one is left to prove that
\[
    \lim_{n\to\infty} \,\frac{\log \tilde{\mu}_{\bf{P}, \bf{i}}(B_n(\omega))}{\sum_{l=1}^n \log b_{i_lj_l}}
    =\lambda({\bf{P}})+ t(\bf{P}) \, \text{ for $\tilde{\mu}_{\bf{P},\bf{i}}$-a.e. $\omega$ and
    ${\bf p}$-a.e. $\bf{i}$}.
\]
It follows from the definition of ${\tilde\mu}_{\bf{P},\bf{i}}$ that, for
${\tilde\mu}_{\bf{P},\bf{i}}$-a.e $\omega$, $p_{i_lj_l}>0$ for
every $l$, so we may restrict our attention to these $\omega$. We have that
\[
  \tilde{\mu}_{\bf{P},\bf{i}}(B_n(\omega))=\prod_{l=1}^n \frac{p_{i_lj_l}}{p_{i_l}}
  \,\prod_{l=1}^{L_n(\omega)}\frac{a_{i_lj_lk_l}^{t(\bf{P})}} {\sum_{k} a_{i_lj_lk}^{t(\bf{P})}}
\]
and
\begin{align*}
  \frac{\log \tilde{\mu}_{\bf{P},\bf{i}}(B_n(\omega))}{\sum_{l=1}^n \log b_{i_lj_l}}&=
  \frac{\sum_{l=1}^n \log \frac{p_{i_lj_l}}{p_{i_l}}} {\sum_{l=1}^n \log b_{i_lj_l}} +
  t({\bf{P}}) \, \frac{\sum_{l=1}^{L_n(\omega)} \log a_{i_lj_lk_l}}
   {\sum_{l=1}^{n} \log b_{i_lj_l}}\\
   &\quad-\frac{\frac{1}{L_n(\omega)}
   \sum_{l=1}^{L_n(\omega)} \log \Bigr(\sum_{k} a_{i_lj_lk}^{t(\bf{P})}
   \Bigl)}{\frac{n}{L_n(\omega)}\frac{1}{n}\sum_{l=1}^{n} \log b_{i_lj_l}}\\
   &= \frac{\alpha_n}{\beta_n}\, + \,t({\bf{P}})  \,  \gamma_n \, - \,\frac{\overset{\,}{\delta_n}}{\theta_n} \, .
\end{align*}
That $\gamma_n\to 1$ follows from (\ref{factor0}). Now we can write
\[
  \alpha_n = \sum_{i,j} \frac{P(\omega,n,(i,j))}{n}\,\log \frac{p_{ij}}{p_i},
\]
where
\[
  P(\omega,n, (i,j))=\sharp \{1\le l\le n: (i_l,j_l)=(i,j)\} .
\]
By Kolmogorov's Strong Law of Large Numbers (KSLLN),
\[
   \frac{P(\omega,n,(i,j)}{n}\to p_{ij} \, \text{ for ${\bf{P}}$-a.e. $\omega$},
\]
so
\[
  \alpha_n \to \sum_{i,j} p_{ij}\log p_{ij} - \sum_i p_i \log p_i \,\text{ for $\bf{P}$-a.e. $\omega$}.
\]
In the same way,
\[
  \beta_n \to \sum_{i,j} p_{ij} \log b_{ij} \,\text{ for $\bf{P}$-a.e. $\omega$},
\]
and, by the definition of $t(\bf{P})$,
\[
   \delta_n\to 0\,\text{ for ${\bf{P}}$-a.e. $\omega$}.
\]
Since $n/L_n^{d-1}(\omega)\ge 1$, we have that $| \theta_n | \ge \log \,(\min\, b_{ij}^{-1})>0$, so we also have that
\[
   \frac{\delta_n}{\theta_n} \to 0\,\text{ for $\bf{P}$-a.e. $\omega$},
\]
thus completing the proof.
\end{proof}
As noticed in the beginning of this Part, these lemmas imply
\[
   \hd\Lambda_{\bf i}\ge\sup_{\bf{P}} \left\{ \lambda({\bf{P}})+t(\bf{P})\right\} \quad \text{for ${\bf p}$-a.e. } \bf{i}.
\]
\text{}

\emph{Part 2:} $\hd\Lambda_{\bf i}\le\sup_{\bf{P}}\{\lambda({\bf{P}})+ t(\bf{P})\}$\\

Let $\underline{t}=\min_{{\bf{P}}} t({\bf{P}})$ and
$\overline{t}=\max_{{\bf{P}}} t({\bf{P}})$. Also let $\mathcal{P}$ be the space of probability vectors ${\bf{P}}=(p_{ij})$
as before (projecting onto $\bf{p}$) such that $p_{ij}>0$ for all $(i,j)$.

\begin{lemma}\label{lemapij}
Given $t\in (\underline{t}, \overline{t})$, there exists a probability vector ${\bf{P}}={\bf{P}}(t)\in\mathcal{P}$,
continuously varying, such that $t({\bf{P}})=t$ and
\begin{equation*}
  {p}_{ij}= p_i\, b_{ij}^{\lambda({\bf{P}})} \Bigl(\sum_{k} a_{ijk}^{t}\Bigr)^{\alpha}\,
  \Bigl(\sum_j b_{ij}^{\lambda({\bf{P}})} \Bigl(\sum_{k} a_{ijk}^{t}\Bigr)^{\alpha}\Bigr)^{-1},
\end{equation*}
where $\alpha=\alpha(t)\in\mathbb{R}$ is $\mathrm{C}^1$.
Moreover, $d \alpha / dt >0$ whenever $\alpha\in [0,1]$, and $\alpha(t)\to-\infty$ when $t\to\underline{t}$ and
$\alpha(t)\to\infty$ when $t\to\overline{t}$.
\end{lemma}
\begin{proof}
Given $\alpha,\lambda\in\mathbb{R}$ and $t\in(\underline{t}, \overline{t})$, we define a
probability vector ${\bf{P}}(\alpha,\lambda,t)\in\mathcal{P}$ by
\begin{equation}\label{pij0}
  p_{ij}(\alpha,\lambda,t)=p_i\, b_{ij}^{\lambda}\, \Bigl(\sum_{k} a_{ijk}^{t}\Bigr)^{\alpha}\,
  \gamma_i(\alpha,\lambda,t)^{-1}
\end{equation}
where
\[
  \gamma_i(\alpha,\lambda,t)= \sum_j b_{ij}^{\lambda}\Bigl(\sum_{k} a_{ijk}^{t}\Bigr)^{\alpha}.
\]

Let $F$ be the continuous function defined by
\begin{equation}\label{F}
   F(\alpha,\lambda,t)=\sum_{i,j} p_{ij}(\alpha,\lambda,t)  \log\Bigl(\sum_k a_{ijk}^t\Bigr).
\end{equation}
We are going to prove there exists a unique $\alpha=\alpha(\lambda,t)$, continuously varying,
such that $F(\alpha,\lambda,t)=0$, i.e. $t({\bf{P}}(\alpha,\lambda,t))=t$.

\emph{Uniqueness.} We have that,
\[
  \frac{\partial p_{ij}}{\partial \alpha}=\log\Bigl(\sum_j a_{ijk}^t\Bigr) p_{ij}
  -\frac{1}{\gamma_i} \frac{\partial \gamma_i}{\partial \alpha} p_{ij}.
\]
Also,
\begin{equation}\label{derivative2}
 \frac{1}{\gamma_i} \frac{\partial \gamma_{i}}{\partial \alpha}=
 \sum_j \frac{p_{ij}}{p_i} \log\Bigl(\sum_k a_{ijk}^t\Bigr).
\end{equation}
So,
\begin{align}
 \frac{\partial F}{\partial \alpha}&=\sum_{i,j} \frac{\partial p_{ij}}{\partial \alpha}
   \log\Bigl(\sum_k a_{ijk}^t\Bigr)\notag\\
   &=\sum_i p_i \left\{ \sum_j \frac{p_{ij}}{p_i} \Bigl(\log\Bigl(\sum_k a_{ijk}^t\Bigr)\Bigr)^2
     -\Bigl(\sum_j \frac{p_{ij}}{p_i} \log\Bigl(\sum_k a_{ijk}^t\Bigr)\Bigr)^2\right\}. \label{dfa}
\end{align}
By the Cauchy-Schwarz inequality we have that the expression
between curly brackets is non-negative and and is positive if there exists $i\in\{1,...,m\}$ such that the function
\[
   j\mapsto \sum_k a_{ijk}^t
\]
is non-constant (note that ${\bf{P}}\in\mathcal{P}$). This is
guaranteed by hypothesis (\ref{hip}). Thus $\partial F/\partial \alpha>0$.

\emph{Existence.} For fixed $(\lambda,t)$, we will look at the limit distributions of
${\bf{P}}(\alpha)={\bf{P}}(\alpha,\lambda,t)$ as $\alpha$ goes to $+\infty$ and $-\infty$.
We see that $\underline{t}$ and $\overline{t}$ are the unique solutions of the following equations, respectively,
\begin{align}
 &\sum_i p_i \log \Bigl( \min_j \sum_k a_{ijk}^{t} \Bigr)=0 \notag\\
 &\sum_i p_i \log \Bigl( \max_j \sum_k a_{ijk}^{t} \Bigr)=0 \label{exist0}.
\end{align}
For instance, if $t_*$ is the solution of (\ref{exist0}) then $t_*=t(\bf{P})$ where
$p_{ij}=p_i \delta_{ij(i)}$ where the maximum appearing in (\ref{exist0}) is attained at $j(i)$,
so $\overline{t}\ge t_*$. On the other hand, for all $\bf{P}$
\[
   \sum_i p_i \sum_j \frac{p_{ij}}{p_i} \log \Bigl( \sum_k a_{ijk}^{t_*} \Bigr)\le
   \sum p_i \log \Bigl( \max_j \sum_k a_{ijk}^{t_*} \Bigr)=0
\]
which implies that $t({\bf{P}})\le t_*$, and so $\overline{t}\le t_*$. Now, for $t\in (\underline{t}, \overline{t})$,
let
\[
    A_{i,t}=\max_{j} \sum_k a_{ijk}^t.
\]
Then
\begin{equation}\label{exist1}
\sum_{i,j} p_{ij}(\alpha)\log\Bigl(\sum_k a_{ijk}^t\Bigr)
  \underset{\alpha\to\infty}{\longrightarrow}\sum_i p_i\log A_{i,t} >0.
\end{equation}
In the same way, defining
\[
    B_{i,t}=\min_{j} \sum_k a_{ijk}^t,
\]
we have
\begin{equation}\label{exist2}
\sum_{i,j} p_{ij}(\alpha)\log\Bigl(\sum_k a_{ijk}^t\Bigr)
  \underset{\alpha\to-\infty}{\longrightarrow}\sum_i p_i\log B_{i,t}<0.
\end{equation}
By (\ref{exist1}), (\ref{exist2}) and continuity, there exists
$\alpha\in\mathbb{R}$ such that
$F(\alpha,\lambda, t)=0$. The continuity of
$\alpha(\lambda, t)$ follows from the uniqueness
part and the implicit function theorem. Actually, since
$F(\alpha,\lambda,t)$ is continuously
differentiable, so is $\alpha(\lambda,t)$.
Observe that
$t({\bf{P}})=\overline{t}\Rightarrow
{\bf{P}}\in\partial\mathcal{P}$ (in this lemma we are assuming
$\underline{t}<\overline{t}$), so since
\[
    t({\bf{P}}(\alpha(\lambda,t)))\to \overline{t} \quad\text{when}\quad
    t\to\overline{t}
\]
then
\[
   {\bf{P}}(\alpha(\lambda,t))\to\partial\mathcal{P}
   \quad\text{when}\quad t\to\overline{t},
\]
which implies
\[
  \alpha(\lambda,t)\to\infty \quad\text{when}\quad t\to\overline{t}
\]
(this convergence is uniform in $\lambda\in [0,1]$).
In the same way we see that
\[
  \alpha(\lambda,t)\to-\infty \quad\text{when}\quad t\to\underline{t}.
\]

We use the following notation $\theta(\lambda,t)=(\alpha(\lambda,t),\lambda,t)$.
We see that
\begin{align*}
 \lambda({\bf{P}}(\theta))=\lambda-
 \frac{\sum_i p_i \log \gamma_i(\theta)}
 {\sum_{i,j} p_{ij} \log b_{ij}}.
\end{align*}
So, we are left to prove there exists
$\lambda=\lambda(t)$, continuously varying, such that
\begin{equation}\label{gama0}
    G(\theta)=\sum_i p_i\log \gamma_i(\theta)=0.
\end{equation}
We have that
\begin{align*}
  \frac{\partial}{\partial \lambda} \sum_i p_i\log \gamma_i(\theta)
  &=\sum_i p_i\frac{1}{\gamma_i(\theta)}\frac{\partial}{\partial \lambda}
  \gamma_i(\theta)\notag \\
  &=\sum_i p_i \left( \frac{1}{\gamma_i(\theta)}\frac{\partial\gamma_i}{\partial \alpha}(\theta)
  \frac{\partial\alpha}{\partial\lambda}+
  \frac{1}{\gamma_i(\theta)}\frac{\partial\gamma_i}{\partial \lambda}(\theta)\right)\notag\\
 &= \sum_i p_i \frac{1}{\gamma_i(\theta)}\frac{\partial\gamma_i}{\partial \lambda}(\theta)
\end{align*}
where we have used (\ref{derivative2}). Now
\begin{equation}\label{gama3}
 \frac{1}{\gamma_i(\theta)}\frac{\partial\gamma_i}{\partial \lambda}(\theta)
 =\sum_j \frac{p_{ij}(\theta)}{p_i} \log b_{ij},
\end{equation}
so
\[
  \frac{\partial}{\partial \lambda} \sum_i p_i\log \gamma_i(\theta)=
  \sum_{i,j} p_{ij}(\theta)\log b_{ij}\le\max_{i,j} \log b_{ij}<0,
\]
and, as before, by the implicit function theorem, this implies
there exists a unique $\lambda=\lambda(t)$, continuously
varying, satisfying (\ref{gama0}).

To conclude the proof we must see that $d\alpha/dt>0$ whenever $\alpha\in [0,1]$.
We have
\[
  \frac{d\alpha}{dt}=\frac{\partial\alpha}{\partial t}+\frac{\partial\alpha}{\partial \lambda}
  \frac{\partial\lambda}{\partial t}
\]
Remember the definition of $F$ from (\ref{F}). Then
\begin{equation*}
   \frac{\partial \alpha}{\partial t}=-\left(\frac{\partial F}{\partial\alpha}\right)^{-1}
   \frac{\partial F}{\partial t}
\end{equation*}
where computation shows that
\begin{align}
  &\frac{\partial F}{\partial t}=\sum_{i,j} p_{ij} \frac{\sum_k a_{ijk}^t \log a_{ijk}}{\sum_k a_{ijk}^t} \label{dft}\\
  &+ \alpha \sum_i p_i \Biggl\{ \sum_j \frac{p_{ij}}{p_i} \frac{\sum_k a_{ijk}^t \log a_{ijk}}{\sum_k a_{ijk}^t}
  \log \Bigl( \sum_k a_{ijk}^t \Bigr)\notag\\
  &\quad\quad\quad\quad\quad\quad\quad
  -\Biggl(\sum_j \frac{p_{ij}}{p_i} \frac{\sum_k a_{ijk}^t \log a_{ijk}}{\sum_k a_{ijk}^t} \Biggr)
  \Biggl( \sum_j \frac{p_{ij}}{p_i} \log \Bigl( \sum_k a_{ijk}^t \Bigr) \Biggr) \Biggr\}\notag.
\end{align}
Also
\[
   \frac{\partial \alpha}{\partial \lambda}=-\left(\frac{\partial F}{\partial\alpha}\right)^{-1}
   \frac{\partial F}{\partial \lambda}
\]
and
\begin{align}
  &\frac{\partial F}{\partial \lambda}=
  \sum_i p_i \Biggl\{ \sum_j \frac{p_{ij}}{p_i} \log b_{ij}  \log \Bigl( \sum_k a_{ijk}^t \Bigr)\label{dfl}\\
  &\quad\quad\quad\quad\quad\quad\quad-\Biggl(\sum_j \frac{p_{ij}}{p_i} \log b_{ij} \Biggr)
  \Biggl( \sum_j \frac{p_{ij}}{p_i} \log \Bigl( \sum_k a_{ijk}^t \Bigr) \Biggr) \Biggr\} \notag.
\end{align}
It follows from (\ref{gama0}) that
\begin{align*}
   \frac{\partial \lambda}{\partial t}&=-\left(\frac{\partial G}{\partial\lambda}\right)^{-1}
   \frac{\partial G}{\partial t}=-\alpha \frac{ \sum_{i,j} p_{ij} \frac{\sum_k a_{ijk}^t \log a_{ijk}}{\sum_k a_{ijk}^t} }
   {\sum_{i,j}p_{ij} \log b_{ij}}.
\end{align*}
Now using hypothesis (\ref{hiptec1}), in the limit case when $\varepsilon\to 0$ we get
\[
   \frac{\partial F}{\partial t}\to \sum_i p_i \log a_i < 0
\]
which implies $\partial \alpha/\partial t>0$ (remember that $\partial F /\partial\alpha>0$).
Also $\partial F/\partial \lambda$ and thus $\partial \alpha / \partial \lambda$ goes to 0, and
$\partial \lambda / \partial t$ is bounded. This implies $d\alpha / dt >0$ which still holds if $\varepsilon>0$
is small enough. Now using hypothesis (\ref{hiptec2}), in the limit case when $\varepsilon\to 0$ we get
\[
   \frac{d\alpha}{d t}\to -\left(\frac{\partial F}{\partial\alpha}\right)^{-1} \sum_{i,j} p_{ij} \log b_{ij}>0,
\]
which still holds if $\varepsilon>0$ is small enough.
\end{proof}

\begin{rem}
We note that hypotheses (1)-(3) were only used in the previous lemma. Namely, hypothesis (1) was used to prove that
$\partial F/\partial \alpha >0$, see (\ref{dfa}), and hypothesis (2) or (3) was used to prove that $d\alpha /d t >0$,
see (\ref{dft}) and (\ref{dfl}). Do we really need these hypotheses?
\end{rem}

Let $s=\sup_{\bf{P}}\{\lambda({\bf{P}})+t({\bf{P}})\}$.
\begin{lemma}\label{lemasup}
For ${\bf p}$-a.e. ${\bf i}$ and for every $\omega\in\Omega_{\bf i}$ there exists ${\bf{P}}\in\mathcal{P}$
such that
\[
 \liminf_{n\to\infty} \frac{\log \tilde{\mu}_{{\bf{P}}, \bf{i}}(B_n(\omega))}
 {\sum_{l=1}^n \log b_{i_lj_l}}\le s.
\]
\end{lemma}
\begin{proof}
Fix ${\bf i}$ and $\omega\in\Omega_{\bf i}$. We use the notation
\[
   d_{{\bf{P}},{\bf i},n}(\omega)=\frac{\log \tilde{\mu}_{{\bf{P}}, {\bf i}}(B_n(\omega))}
   {\sum_{l=1}^n \log b_{i_lj_l}}.
\]
Then it follows from the proof of Lemma \ref{dimmedida} that, if ${\bf{P}}\in\mathcal{P}$,
\begin{align}
  d_{{\bf{P}},{\bf i},n}(\omega)=&
  \frac{\sum_{l=1}^{n}\log p_{i_lj_l}-\sum_{l=1}^{n}
  \log p_{i_l}}{\sum_{l=1}^{n} \log b_{i_lj_l}} \label{vert1}\\
  &+\eta_n(\omega) t({\bf{P}})-\frac{\sum_{l=1}^{L_n(\omega)}
  \log\Bigl(\sum_k a_{i_lj_lk}^{t({\bf{P}})}\Bigr)}{\sum_{l=1}^n \log b_{i_lj_l}}\notag
\end{align}
where
\[
  \eta_n(\omega)=\frac{\sum_{l=1}^{L_n(\omega)} \log a_{i_lj_lk_l}}{\sum_{l=1}^n \log b_{i_ljl}}
  \underset{n\to\infty}{\longrightarrow} 1.
\]
Given $t\in (\underline{t}, \overline{t})$,
consider the probability vector ${\bf{P}}(t)$, such that
$t({\bf{P}}(t))=t$, given by Lemma \ref{lemapij}. Applying
(\ref{vert1}) to ${\bf{P}}(t)$ we obtain
\begin{align}
  &d_{{\bf{P}}(t),{\bf i},n}(\omega)=
 \lambda({\bf{P}}(t)) +\eta_n(\omega) t\label{vert2}-
 \frac{\sum_{l=1}^n \log \gamma_{i_l}(t)}{\sum_{l=1}^n \log b_{i_lj_l}}\\
  &+\frac{\alpha(t) \sum_{l=1}^{n} \log \Bigl(\sum_k a_{i_lj_lk}^{t}\Bigr)
  -\sum_{l=1}^{L_n(\omega)} \log \Bigl(\sum_k a_{i_lj_lk}^{t}\Bigr)}
  {\sum_{l=1}^n \log b_{i_lj_l}}, \notag
\end{align}
where, by KSLLN and (\ref{gama0}),
\[
   \frac{1}{n}\sum_{l=1}^n \log \gamma_{i_l}(t)\to \sum_i p_i \log \gamma_i(t)=0\quad\text{for $\bf{p}$-a.e. ${\bf i}$}.
\]
So we must prove that there exists $t_*\in (\underline{t},
\overline{t})$ such that
\begin{equation}\label{princ}
\limsup_{n\to\infty} \frac{1}{n} \Bigl\{
  \alpha(t_*) \sum_{l=1}^{n}
  \log \Bigl(\sum_k a_{i_lj_lk}^{t_*}\Bigr)
  -\sum_{l=1}^{L_n(\omega)} \log \Bigl(\sum_k a_{i_lj_lk}^{t_*}\Bigr) \Bigr\}\ge 0.
\end{equation}

By Lemma \ref{lemapij} and the inverse function theorem, given
$a\in [0,1]$, there exists a unique
function $t(a)\in (\underline{t}, \overline{t}) $, which is
continuous, increasing in $a$ and satisfies
\begin{equation} \label{alfatro}
    \alpha(t(a))=a.
\end{equation}
Let
\[
    a_0=\liminf_{n\to\infty}\frac{L_n(\omega)}{n},\quad
    a_1=\limsup_{n\to\infty}\frac{L_n(\omega)}{n},
\]
and
\[
  t_0=t(a_0),\quad t_1=t(a_1).
\]
Let
\[
   a_n=\frac{L_n(\omega)}{n} \quad\text{and}\quad t_n=t(a_n).
\]
Then, by Lemma \ref{calclema} (and Remark \ref{remcalc}), for every $t\in [t_0, t_1]$
\begin{equation}\label{F2}
F(t)=\limsup_{n\to\infty} \frac{1}{n} \Bigl\{ a_n \sum_{l=1}^{n}
  \log \Bigl(\sum_k a_{i_lj_lk}^{t}\Bigr)
  -\sum_{l=1}^{L_n(\omega)} \log \Bigl(\sum_k a_{i_lj_lk}^{t}\Bigr) \Bigr\}\ge 0.
\end{equation}
Since, for all $(i,j)$,
\begin{equation}\label{fcont}
     t\mapsto \log \Bigl(\sum_k a_{ijk}^{t}\Bigr)
\end{equation}
are continuous functions, so is $F(t)$. By adding some
constant, we may assume the functions in (\ref{fcont}) are $\ge
1$, because by definition of $a_n$ this does not
change $F(t)$. Note that, by (\ref{alfatro}), $a_n=\alpha(t_n)$, and let $\bar{t}(t)$ be the biggest accumulation point of
$(t_n)$ for which the $\limsup$ in (\ref{F2}) is attained. The continuity of $F$ and the functions in (\ref{fcont}) imply
that $\bar{t}(t)$ is also continuous. So
\begin{align}
F(t)=\limsup_{n\to\infty} \frac{1}{n} \Bigl\{
  \alpha(\bar{t}(t)) \sum_{l=1}^{n}
  \log \Bigl(\sum_k a_{i_lj_lk}^{t}\Bigr)
  -\sum_{l=1}^{L_n(\omega)} \log \Bigl(\sum_k a_{i_lj_lk}^{t}\Bigr) \Bigr\}.\notag
\end{align}
Since
\[
  \bar{t}\colon [t_0, t_1] \to [t_0, t_1]
\]
is continuous, by Brouwer's fixed point theorem there is
$t_*\in [t_0, t_1]$ such that $\bar{t}(t_*)=t_*$, thus proving (\ref{princ}).
\end{proof}

Part 2 will be concluded in the following lemma.
\begin{lemma} For ${\bf p}$-a.e. ${\bf i}$,
\begin{equation*}
   \hd\Lambda_{\bf i}\le\sup_{\bf{P}}\{\lambda({\bf{P}})+ t(\bf{P})\}.
\end{equation*}
\end{lemma}
\begin{proof}
Let $\bf i$ be as in Lemma \ref{lemasup}. Let $\varepsilon>0$. Consider the \emph{approximate squares of order
$n$} given by $B_n(z)=\chi_{\bf i}(B_n(\omega))$ where $\omega\in \chi_{\bf i}^{-1}(z)$, $z\in\Lambda_{\bf i}$,
$n\in\mathbb{N}$.
Then it follows from Lemma \ref{lemasup} that
\begin{equation}\label{dimp}
  \forall_{z\in\Lambda_{\bf i}} \;\forall_{N\in\mathbb{N}}\;
  \exists_{n>N}\;\exists_{{\bf{P}}\in\mathcal{P}} :\,
  \frac{\log \mu_{{\bf{P}}, {\bf i}}(B_n(z))} {\log |B_n(z)|}\le s+\varepsilon.
\end{equation}
Given $\delta, \eta>0$, we shall build a cover
$\mathcal{U}_{\delta,\eta}$ of $\Lambda_{\bf i}$ by sets with diameter
$<\eta$ such that
\[
    \sum_{U\in\mathcal{U}_{\delta,\eta}} |U|^{s+\varepsilon+2\delta}\le
    \sqrt{2}\,\max a_{ijk}^{-1}\, M_\delta
\]
where $M_\delta$ is an integer depending on $\delta$ but not on
$\eta$. This implies that $\hd\Lambda_{\bf i}\le s+\varepsilon+2\delta$
which gives what we want because $\varepsilon$ and $\delta$ can be
taken arbitrarily small. Let $b=\max\,b_{ij}<1$. It is clear, using compactness arguments and the continuity of
${\bf P}\mapsto t({\bf P})$, that there exists a finite number of Bernoulli measures
$\mu_1,...,\mu_{M_\delta}$ such that
\[
      \forall_{{\bf{P}}}\; \exists_{k\in\{1,...,M_\delta\}} :\,
      \frac{\mu_{{\bf{P}}, {\bf i}}(B_n)}{\mu_k(B_n)}\le b^{-\delta n}
\]
for all approximate squares of order $n$, $B_n$. By (\ref{dimp}), we
can build a cover of $\Lambda_{\bf i}$ by approximate squares
$B_{n(z^l)},\,l=1,2,...$ that are disjoint and have diameters
$<\eta$, such that
\[
    \mu_{{\bf{P}}^l, {\bf i}}(B_{n(z^l)})\ge |B_{n(z^l)}|^{s+\varepsilon+\delta}
\]
for some probability vectors ${\bf{P}}^l$. It follows that
\begin{align*}
  \sum_l |B_{n(z^l)}|^{s+\varepsilon+2\delta} &\le \sum_l \mu_{{\bf{P}}^l, {\bf i}}
   (B_{n(z^l)})\, |B_{n(z^l)}|^\delta\\
  &\le \sum_l \mu_{k_l}(B_{n(z^l)}) \,b^{-\delta n(z^l)}
  \, \sqrt{2}\,\max a_{ijk}^{-1}\, b^{\delta n(z^l)} \\
  &\le \sqrt{2}\,\max a_{ijk}^{-1}\,\sum_{k=1}^{M_\delta} \sum_l \mu_k(B_{n(z^l)})
  \le\sqrt{2}\,\max a_{ijk}^{-1}\, M_\delta
\end{align*}
as we wish.
\end{proof}

This ends the proof of Theorem A.

\section{A calculus lemma}

\begin{lemma}\label{calclema}
Let $f\colon (0,\infty)\to\mathbb{R}$ be a Lipschitz function and $\alpha\colon (0,\infty)\to\mathbb{R}$ a positive
bounded function. Then
\begin{equation}\label{kp0}
   \limsup_{u\to\infty} \frac{1}{u} \Bigl(\alpha(u) f(u)-f(\alpha(u) u)\Bigr)\ge 0.
\end{equation}
\end{lemma}
\begin{proof}
Let $g\colon (0,\infty)\to\mathbb{R}$ be defined by $g(x)=e^{-x} f(e^x)$. Then $g$ is bounded and we must see that
\begin{equation}\label{kp3}
 \limsup_{x\to\infty} \alpha(e^x) \Bigl( g(x)-g(x+\log \alpha(e^x))\Bigr)\ge 0.
\end{equation}
Just take a sequence $x_n\to\infty$ such that
\[
 \limsup_{x\to\infty} g(x)=\lim_{n\to\infty} g(x_n).
\]
\end{proof}

\begin{rem}\label{remcalc}
Lemma \ref{calclema} also works when the functions $f$ and $\alpha$ are defined only on the
positive integers, by extending them in a piecewise linear
fashion, if we substitute $f$ being Lipschitz by
\[
   |f(n+1)-f(n)|\le C
\]
for all $n$, for some constant $C>0$.
\end{rem}

In what follows we make some extensions of Lemma \ref{calclema} which are not used in this paper but should be useful
when one tries to extend this paper to higher dimensions (see the problem proposed at the end of this section).

The next lemma is a non-linear extension of [7, Lemma 4.1] which was used to compute the Hausdorff dimension of
multidimensional versions of general Sierpi\'nski carpets.

\begin{lemma}\label{lemakp}
Let $f_k\colon (0,\infty)\to\mathbb{R}$ be Lipschitz functions for
$k=1,2,...,r$, and suppose $\alpha_k\colon
(0,\infty)\to\mathbb{R}$ is bounded, $\mathrm{C}^1$ and there
exist positive constants $\delta, C$ such that
\begin{align}
  &\bullet \alpha_k(u)>\delta, \text{ for } u>0\label{kp1}\\
  &\bullet |u \,\alpha_k'(u)|\to 0 \text{ as } u\to\infty.    \label{kp2}
\end{align}
Then
\begin{equation}\label{kp0}
   \limsup_{u\to\infty} \frac{1}{u} \sum_{k=1}^r
   \Bigl(\alpha_k(u) f_k(u)-f_k(\alpha_k(u) u)\Bigr)\ge 0.
\end{equation}
\end{lemma}
\begin{proof}
As before and following \cite{7}, we define $g_k\colon (0,\infty)\to\mathbb{R}$
by $g_k(x)=e^{-x} f_k(e^x)$ for $k=1,...,r$. Then we must see that
\begin{equation}\label{kp3}
 \limsup_{x\to\infty} \sum_{k=1}^r \alpha_k(e^x) \Bigl( g_k(x)-g_k(x+\log \alpha_k(e^x))\Bigr)
 \ge 0.
\end{equation}
We will see that
\begin{equation}\label{kp4}
\left| \int_{u}^{b(u)} \sum_{k=1}^r \alpha_k(e^x) \Bigl(
g_k(x)-g_k(x+\log \alpha_k(e^x))\Bigr)\, dx
 \right|
\end{equation}
is bounded in $u$, for some $b(u)>u$ with $b(u)-u\to\infty$ as $u\to\infty$, which implies (\ref{kp3}).

Let $\xi(u)$ be a decreasing function converging to 0 as $u\to\infty$ such that $|u\, \alpha_k'(u)|\le \xi(u)$ for $u>0$
($\xi(u)=\max\{ |x\, \alpha_k'(x)| : x\ge u\}$ will do). Then using (\ref{kp2}) and
\[
   b(u)=u+\xi(e^u)^{-1}
\]
we get that
\begin{equation}\label{kp9}
 \int_u^{b(u)} |\alpha_k'(e^x)\,e^x|\le 1.
\end{equation}
Note that the functions $g_k$ are bounded because the functions $f_k$ are Lipschitz.
Using the intermediate value theorem and (\ref{kp9}) we obtain
\[
 \left|\sum_{k=1}^r \int_u^{b(u)}  \Bigl(\alpha_k(e^x)-\alpha_k(e^{x +\log \alpha_k(e^x)})\Bigr)
  g_k(x+\log \alpha_k(e^x))\,dx \right|\le M
\]
for $u>0$ and some $M>0$. So, (\ref{kp4}) is bounded by $M$ plus
\begin{equation}\label{kp5}
 \left| \int_u^{b(u)} \sum_{k=1}^r \Bigl(\alpha_k(e^x) g_k(x)-
 \alpha_k(e^{x +\log \alpha_k(e^x)}) g_k(x+\log \alpha_k(e^x)) \Bigr)\, dx \right|.
\end{equation}
By doing the change of coordinates $y=x+\log \alpha_k(e^x)$, which is invertible for $x>a$, for some $a>0$,
due to conditions (\ref{kp1}) and (\ref{kp2}), (\ref{kp5}) becomes
\begin{align}
 &\Biggl|  \sum_{k=1}^r \Biggl( \int_u^{b(u)} \alpha_k(e^x) g_k(x)\,dx \notag \\
 &\quad\quad\quad\quad\quad
   -\int_{u+\log \alpha_k(e^u)}^{b(u)+\log \alpha_k(e^{b(u)})}
   \alpha_k(e^y) g_k(y) \frac{\alpha_k(e^{x(y)})}{\alpha_k(e^{x(y)})+
   \alpha_k'(e^{x(y)}) e^{x(y)}}\,dy \Biggr)\Biggr|\notag\\
 & \le\sum_{k=1}^r \Biggl|  \int_u^{u+\log \alpha_k(e^u)} \alpha_k(e^x) g_k(x)\,dx \Biggr|+
  \sum_{k=1}^r \Biggl|  \int_{b(u)+\log\alpha_k(e^{b(u)})}^{b(u)}  \alpha_k(e^x) g_k(x)\,dx \Biggr|
  \label{kp6}\\
 &+\sum_{k=1}^r \Biggl|  \int_{u+\log \alpha_k(e^u)}^{b(u)+\log \alpha_k(e^{b(u)})}
 \alpha_k(e^y) g_k(y) \frac{\alpha_k'(e^{x(y)}) e^{x(y)}}
 {\alpha_k(e^{x(y)})+\alpha_k'(e^{x(y)})e^{x(y)}}\,dy \Biggr|\label{kp7}.
\end{align}
The terms in (\ref{kp6}) are bounded because the functions $g_k$
and $\alpha_k$ are bounded. That (\ref{kp7}) is also bounded follows from (\ref{kp9}). Thus (\ref{kp4}) is bounded,
concluding the proof.
\end{proof}

\begin{corollary}\label{corkp}
Suppose in addition to Lemma \ref{lemakp} hypotheses there are
functions $\beta_k\colon (0,\infty)\to\mathbb{R},\,k=1,...,r$
satisfying the same hypotheses of $\alpha_k$. Then
\begin{equation*}
   \limsup_{u\to\infty} \frac{1}{u} \sum_{k=1}^r
   \Bigl(\frac{\alpha_k(u)}{\beta_k(u)} f_k(\beta_k(u) u)-f_k(\alpha_k(u) u)\Bigr)\ge 0.
\end{equation*}
\end{corollary}
\begin{proof}
Define $g_k\colon (0,\infty)\to \mathbb{R}$ by
$g_k(u)=f_k(\beta_k(u) u)$. Since the functions $\beta_k$ satisfy
the same conditions (\ref{kp1}) and (\ref{kp2}) as $\alpha_k$, we
can easily see that the functions $\frac{\alpha_k}{\beta_k}$ also
do satisfy them, and that $g_k$ are Lipschitz functions. Since
\[
  \frac{\alpha_k(u)}{\beta_k(u)} f_k(\beta_k(u) u)-f_k(\alpha_k(u) u)=
  \frac{\alpha_k(u)}{\beta_k(u)} g_k(u)-g_k(\frac{\alpha_k(u)}{\beta_k(u)} u),
\]
we can apply Lemma \ref{lemakp}.
\end{proof}

Unfortunately, Lemma \ref{lemakp} does not hold when we substitute hypothesis (\ref{kp2}) by the weaker one
\begin{equation}\label{weaker}
 |u \,\alpha_k'(u)|\le C
\end{equation}
for $u>0$, for some constant $C>0$. This is shown in the next example which was kindly communicated to me by Gustavo
Moreira (Gugu).

\begin{example}\label{contraexemplo}
Let $h(x)$ and $c(x)$ be bounded $\mathrm{C}^1$ functions with bounded derivative and $c(x)>b$ for all $x>0$, for some
$b>0$. Then $f(x)=x h(\log x)$ is a Lipschitz function and $\alpha(x)=c(\log x)$ satisfies hypotheses (\ref{kp1}) and
(\ref{weaker}). We have that
\[
  \alpha(u)f(u)-f(\alpha(u) u)=u c(\log u) (h(\log u)-h(\log u + \log c(\log u))).
\]
If we put $h(x)=\sin x$ and $c(x)=e^{\cos x}$, we have that $\alpha(u)f(u)-f(\alpha(u) u)\le 0$ for all $u>0$. Taking
$h_k(x)=h(x+k\pi/2),\,c_k(x)=c(x+k\pi/2),\,f_k(x)=x h_k(\log x)$ and $\alpha_k(x)=c_k(\log x)$ for $k=1,2$, we have that
\[
  \frac{1}{u} \sum_{k=1}^r  \Bigl(\alpha_k(u) f_k(u)-f_k(\alpha_k(u) u)\Bigr) \le -\delta
\]
for all $u>0$, for some $\delta>0$.
\end{example}
\text{}

\emph{Problem:} Does Lemma \ref{lemakp} hold with hypothesis (\ref{weaker}) instead of (\ref{kp2}), for \emph{generic
$f_k$ and $\alpha_k$} ?
\\\\
If the answer to this problem is affirmative in some sense then we believe we can compute the Hausdorff dimension of
\emph{generic self--affine Sierpi\'nski sponges} which are the 3-dimensional versions of the self--affine Sierpi\'nski
carpets.

\text{}\\

\textbf{Acknowledgments:} I wish to thank Gustavo Moreira for communicating to me Example \ref{contraexemplo}.
This work was supported by Fundação para a Ciência e a Tecnologia (Portugal).

\end{document}